\documentclass[twoside,a4paper,reqno,11pt]{amsart} 
\usepackage[top=28mm,right=30mm,bottom=28mm,left=30mm]{geometry}

\usepackage{amsfonts, amsmath, amssymb, mathrsfs, bm, latexsym, stmaryrd, array, hyperref, mathtools, bold-extra, xcolor}

\renewcommand{\a}{\alpha}

\renewcommand{\O}{\Omega}

\newcommand{\la}{\langle}
\newcommand{\ra}{\rangle}

\renewcommand{\to}{\rightarrow}

\newcommand{\leqs}{\leqslant}
\newcommand{\geqs}{\geqslant}

\newcommand{\vs}{\vspace{2mm}}

\makeatletter
\newcommand{\imod}[1]{\allowbreak\mkern4mu({\operator@font mod}\,\,#1)}
\makeatother

\newtheorem{theorem}{Theorem} 
\newtheorem*{conj*}{Conjecture}

\newtheorem{thm}{Theorem}[section] 
\newtheorem{prop}[thm]{Proposition} 
\newtheorem{lem}[thm]{Lemma}

\theoremstyle{definition}
\newtheorem{rem}[thm]{Remark}
\newtheorem{remk}{Remark}
\newtheorem*{definition}{Definition}

\begin{document}

\author{Timothy C. Burness}
\address{T.C. Burness, School of Mathematics, University of Bristol, Bristol BS8 1UG, UK}
\email{t.burness@bristol.ac.uk}

\author{Emily V. Hall}
\address{E.V. Hall, School of Mathematics, University of Bristol, Bristol BS8 1UG, UK}
\email{ky19128@bristol.ac.uk}

\title[Almost elusive groups]{Almost elusive permutation groups} 

\begin{abstract}
Let $G$ be a nontrivial transitive permutation group on a finite set $\Omega$. An element of $G$ is said to be a derangement if it has no fixed points on $\Omega$. From the orbit counting lemma, it follows that $G$ contains a derangement, and in fact $G$ contains a derangement of prime power order by a theorem of Fein, Kantor and Schacher. However, there are groups with no derangements of prime order; these are the so-called  elusive groups and they have been widely studied in recent years. Extending this notion, we say that $G$ is almost elusive if it contains a unique conjugacy class of derangements of prime order. In this paper we first prove that every quasiprimitive almost elusive group is either almost simple or $2$-transitive of affine type. We then classify all the almost elusive groups that are almost simple and primitive with socle an alternating group, a sporadic group, or a rank one group of Lie type. 
\end{abstract}

\date{\today}

\maketitle

\section{Introduction}\label{s:intro}

Let $G \leqs {\rm Sym}(\O)$ be a transitive permutation group on a finite set $\O$ with $|\O| \geqs 2$ and point stabiliser $H$. By a classical theorem of Jordan \cite{Jordan}, which is an easy consequence of the orbit counting lemma, $G$ contains elements that act fixed point freely on $\O$. Such an element is called a \emph{derangement} and we note that $x \in G$ has this property if and only if $x^G \cap H$ is empty, where $x^G$ denotes the conjugacy class of $x$. In particular, the set of derangements is closed under conjugation. Derangements arise naturally in a wide range of contexts and Jordan's theorem turns out to have interesting applications in several different areas (see Serre's article \cite{serre}, for example).

The existence of derangements leads to a number of natural problems that have been extensively studied by various authors. For example, there is a substantial literature on the proportion of derangements in finite transitive permutation groups. Here one of the main highlights is the sequence of papers \cite{FG1,FG2,FG3,FG4} by Fulman and Guralnick, which shows that the proportion of derangements in a transitive simple group is bounded from below by an absolute constant (this settles a conjecture of Boston and Shalev from the 1990s). 

In this paper, we focus on the existence of derangements with prescribed properties, noting that problems of this flavour have also attracted significant interest in recent years. A landmark result in this direction is established by Fein, Kantor and Schacher in \cite{FKS}. By applying the  Classification of Finite Simple Groups, they prove that every nontrivial finite transitive group contains a derangement of prime power order. Moreover, they also observe that the conclusion does not extend to prime order derangements, in general. For example, the $3$-transitive action of the smallest Mathieu group ${\rm M}_{11}$ on $12$ points has no derangements of prime order (but it does contain derangements of order $4$ and $8$). Indeed, ${\rm M}_{11}$ has unique conjugacy classes of elements of order $2$ and $3$, and both primes divide the order of a point stabiliser ${\rm L}_{2}(11)$.

Following \cite{CGJKKMN}, we say that a transitive group is \emph{elusive} if it contains no derangements of prime order. These groups have been the subject of several papers in recent years (see \cite{CGJKKMN, G0, G, GK, GMPV, Xu} for example), but a complete classification remains out of reach. One of the main results towards a classification is a theorem of Giudici \cite{G}, which states that if $G$ is an elusive group with a transitive minimal normal subgroup, then there exists a positive integer $k$ such that $G = {\rm M}_{11} \wr A$ in its product action on $\Delta^k$, where $|\Delta|=12$ and $A \leqs S_k$ is transitive. In particular, every elusive group with this property is primitive. 

Further interest in elusive groups stems from an open problem in algebraic graph theory from the early 1980s. In \cite{Maru}, Maru\v{s}i\v{c} conjectures that if $\Gamma$ is a finite vertex-transitive digraph, then ${\rm Aut}(\Gamma)$ contains a derangement of prime order (with respect to the action on vertices). This was later extended by Klin (see \cite[Problem 282 (BCC15.12)]{klin}), who conjectures that  the same conclusion holds for every nontrivial finite transitive $2$-closed permutation group (it is easy to see that ${\rm Aut}(\Gamma)$ as above is $2$-closed, so this is a natural generalisation). This is known as the \emph{Polycirculant Conjecture} and although there has been progress towards a positive solution, both problems remain open (see \cite[Section 1.3.4]{BG_book} for further details and references). In particular, none of the known elusive groups are $2$-closed.

In this paper, we introduce and study a new family of permutation groups. 

\begin{definition}
Let $G \leqs {\rm Sym}(\O)$ be a permutation group. Then $G$ is \emph{almost elusive} if it contains a unique conjugacy class of derangements of prime order. 
\end{definition}

For example, if $n = p^a$ is a prime power then it is easy to see that the natural action of the symmetric group $S_n$ on $n$ points is almost elusive (every derangement of prime order is a product of $n/p$ disjoint $p$-cycles, which form a single conjugacy class). In particular, there are infinitely many almost simple primitive groups with this property, which is in stark contrast to the situation for elusive groups, where Giudici's theorem \cite{G} implies that the action of ${\rm M}_{11}$ on $12$ points is the only example. We can also find affine type examples. For instance, if $q = 2^f$ with $f \geqs 1$ then the natural $2$-transitive action of ${\rm AGL}_{2}(q)$ on $q^2$ points is almost elusive.

\begin{remk}\label{r:single}
As noted above, Jordan's theorem implies that every nontrivial finite transitive group  contains at least one conjugacy class of derangements. It turns out that there are groups with a unique conjugacy class of derangements. Indeed, the main theorem of \cite{BTV1} states that a primitive group $G \leqs {\rm Sym}(\O)$ with point stabiliser $H$ has a unique class of derangements if and only if $G$ is sharply $2$-transitive (that is, any pair of distinct elements in $\O$ can be mapped to any other such pair by a unique element in $G$) or $(G,H) = (A_5, D_{10})$ or $({\rm L}_{2}(8){:}3, D_{18}{:}3)$. Further work by Guralnick \cite{Gur} shows that the same conclusion holds for all transitive groups. Note that all of these groups are almost elusive.
\end{remk}

\begin{remk}\label{r:relusive}
Let $G \leqs {\rm Sym}(\O)$ be a nontrivial finite transitive permutation group with point stabiliser $H$ and let $r$ be a prime divisor of $|\O|$. Following \cite{BGW}, we say that $G$ is \emph{$r$-elusive} if $G$ does not contain a derangement of order $r$, whence $G$ is elusive if and only if $G$ is $r$-elusive for every prime divisor $r$ of $|\O|$. Similarly, $G$ is almost elusive only if the same conclusion holds for all but one prime $r$. In particular, note that $G$ is almost elusive only if $|\pi(G) \setminus \pi(H)| \leqs 1$, where $\pi(X)$ denotes the set of prime divisors of $|X|$. We refer the reader to \cite{BGiu, BG_book, BGW} for results on $r$-elusive primitive groups.  
\end{remk}

Recall that a finite permutation group is \emph{quasiprimitive} if every nontrivial normal subgroup is transitive. In \cite{P93}, Praeger establishes a version of the O'Nan-Scott theorem for quasiprimitive groups, which describes the structure and action of such a group in terms of its socle (recall that the \emph{socle} of a group is the product of its minimal normal subgroups). By applying this important theorem, we can prove the following result.

\begin{theorem}\label{t:main1}
Let $G$ be a finite quasiprimitive almost elusive permutation group. Then either $G$ is almost simple, or $G$ is a $2$-transitive affine group. 
\end{theorem} 

\begin{remk}\label{r:1}
It is worth noting that there exist almost elusive quasiprimitive groups that are not primitive (once again, this differs from the situation for elusive groups). For instance, suppose $G = {\rm L}_{2}(q)$ and $q=2^m-1$ is a Mersenne prime such that $2^{m-1}-1$ is divisible by $9$. For example, we can take
\[
m \in \{ 7, 13, 19, 31, 61, 127, 607,1279, 2203, 2281, 3217, 4423,  \ldots \}.
\]
Let $H = C_q{:}C_{(q-1)/2}$ be a Borel subgroup of $G$ and set $\O = G/K$, where $K = C_q{:}C_{(q-1)/6}$ is a subgroup of $H$. Then $G$ is quasiprimitive (but not primitive) on $\O$ and we note that $|\O| = 3.2^m$. Moreover, $G$ has unique conjugacy classes of elements of order $2$ and $3$, and our choice of $q$ implies that $|K|$ is divisible by $3$. Therefore, $G$ is almost elusive on $\O$. We do not know if there are infinitely many almost elusive quasiprimitive groups that are not primitive.
\end{remk}

With the reduction theorem in hand, our ultimate aim is to classify all the almost elusive quasiprimitive groups. In this paper, we take a first step in this direction by establishing Theorem \ref{t:main2} below on almost simple primitive groups (recall that $G$ is \emph{almost simple} if the socle $G_0$ of $G$ is a nonabelian finite simple group, in which case $G_0 \leqs G \leqs {\rm Aut}(G_0)$). In order to state this result, set
\[
\mathcal{G} = \mathcal{A} \cup \mathcal{B},
\]
where $\mathcal{A}$ is the set of all alternating groups $A_n$ with $n \geqs 5$, and $\mathcal{B}$ is the set of all sporadic simple groups (including the Tits group ${}^2F_4(2)'$), together with all simple groups of Lie type of the form ${\rm L}_{2}(q)$ (with $q \geqs 7$ and $q \ne 9$), ${\rm U}_{3}(q)$ (with $q \geqs 3$), ${}^2G_2(q)$ (with $q \geqs 27$) and ${}^2B_2(q)$ (with $q \geqs 8$). Note that $\mathcal{G}$ contains every simple group of Lie type with (twisted) Lie rank equal to $1$.

\begin{theorem}\label{t:main2}
Let $G \leqs {\rm Sym}(\O)$ be a finite almost simple primitive permutation group with socle $G_0 \in \mathcal{G}$ and point stabiliser $H$. Then $G$ is almost elusive if and only if $(G,H)$ is one of the cases recorded in Table \ref{tab:main1} or \ref{tab:main2}.
\end{theorem}

\begin{table}
\[
\begin{array}{lllll} \hline
G_0 & G & H & \mbox{Conditions} & x \\ \hline
A_n & S_n & S_{n-1} & n = r^a & [r^{n/r}] \\
& & S_{n-2} \times S_2 & n = 2^m = r + 1 & [r,1] \\
& & & n = 2^m+1 = r & [r] \\
& A_n & A_{n-1} & \mbox{$n = r^a$, $a \geqs 2$} & [r^{n/r}] \\
& & & n=2r^a,\, r \geqs 3 & [r^{n/r}] \\
& & & & \\
A_{10} & A_{10} & (S_7 \times S_3) \cap G & & [5^2] \\
A_9 & S_9, A_9 & (S_7 \times S_2) \cap G & & [3^3] \\
& & (S_6 \times S_3) \cap G & & [7,1^2] \\
A_6 & S_6 & S_3 \wr S_2 & &  [5,1] \\
& A_6 & {\rm L}_{2}(5) & & [3,1^3] \\
& {\rm PGL}_{2}(9) & D_{20} & & 3 \\
& {\rm M}_{10} & 5{:}4 & & 3 \\
& & 3^2{:}Q_8 & & 5 \\
A_5 & A_5 & D_{10} & & [3,1^2] \\ \hline
\end{array}
\]
\caption{The primitive almost elusive groups with socle $A_n$, $n \geqs 5$}
\label{tab:main1}
\end{table}

\begin{table}
\[
\begin{array}{lllll} \hline
G_0 & \mbox{Type of $H$} & G & \mbox{Conditions} & x \\ \hline
{\rm L}_{2}(q) & P_1 & {\rm PGL}_{2}(q), \, G_0  & q = p = 2^m-1 & 2 \\
&  & G_0 & q=p, \, p+1 = 2.3^a, \, a \geqs 2 & 3 \\
& & G_0.f &  \mbox{See Remark \ref{r:2}(f)} & r \\
& & G_0.3, \, G_0 & q = 8 & 3 \\
& {\rm GL}_{1}(q) \wr S_2 & {\rm PGL}_{2}(q) &  q=p=2^m-1 & p \\ 
& & G_0.3, \, G_0 & q = 8 & 3 \\
& {\rm GL}_{1}(q^2) & {\rm PGL}_{2}(q) & q=p=2^m+1 & p \\ 
& & G_0.3 & q= 8 & 7 \\ 
{\rm U}_{3}(q) & P_1 & G_0.2 & q=3 & 7 \\
& {\rm GU}_{2}(q) \times {\rm GU}_{1}(q) & G_0.4 & q = 4 & 13 \\ 
& & G_0.6 & q = 8 & 19 \\
& {\rm GU}_{1}(q) \wr S_3 & G_0.4 & q = 4 & 13 \\
& {\rm L}_{2}(7) & G_0.2, \, G_0 & q=3 & [J_2,J_1] \\ 
{}^2F_4(2)' & {\rm L}_{2}(25) & G_0.2,\, G_0 & & \texttt{2A} \\
& 5^2{:}4A_4 & G_0.2 & & 13 \\ \hline
\end{array}
\]
\caption{The primitive almost elusive groups with socle $G_0 \in \mathcal{B}$}
\label{tab:main2}
\end{table}

\begin{remk}\label{r:2}
Some comments on the statement of Theorem \ref{t:main2} are in order.
\begin{itemize}\addtolength{\itemsep}{0.2\baselineskip}
    \item[{\rm (a)}] If $G_0$ is one of the $26$ sporadic simple groups then by combining Theorem \ref{t:main2} with the main result of \cite{G} we deduce that either $(G,H) = ({\rm M}_{11},{\rm L}_{2}(11))$ and $G$ is elusive, or $G$ contains at least two conjugacy classes of derangements of prime order. In particular, $G$ is not almost elusive. Similarly, there are no primitive almost elusive groups with socle ${}^2G_2(q)$ (with $q \geqs 27$) or ${}^2B_2(q)$.
    \item[{\rm (b)}] In the fourth column of Table \ref{tab:main1}, $r$ denotes a prime number and $a$ is a positive integer. For example, in the second row $r = 2^m-1$ is a Mersenne prime and thus $m$ is a prime. Similarly, in the next row $r=2^m+1$ is a Fermat prime, which implies that $m$ is a $2$-power. 
  \item[{\rm (c)}] In the final column of Table \ref{tab:main1} we record a representative $x$ of the unique conjugacy class in $G$ of derangements of prime order; in the cases where $G = S_n$ or $A_n$, we give the cycle-shape of $x$ in the form $[r^d,1^{n-dr}]$, which means that $x$ is a product of $d$ disjoint $r$-cycles in its natural action on $\{1, \ldots, n\}$. For the cases with $G = {\rm PGL}_{2}(9)$, we use $3$ to denote a representative in the unique conjugacy class of elements of order $3$ in $G$. Similarly, we write $5$ for the unique class of elements of order $5$ in ${\rm M}_{10}$.
     \item[{\rm (d)}] In Table \ref{tab:main2} we list all the primitive almost elusive groups with socle $G_0 \in \mathcal{B}$. Note that the conditions on $q$ in the definition of $\mathcal{B}$ are justified in view of the isomorphisms
\[
{\rm L}_{2}(4) \cong {\rm L}_{2}(5) \cong A_5,\;\; {\rm L}_{2}(9) \cong A_6, \;\; {}^2G_2(3)' \cong {\rm L}_2(8),
\]
together with the fact that the groups ${\rm L}_{2}(2)$, ${\rm L}_{2}(3)$, ${\rm U}_{3}(2)$ and ${}^2B_2(2)$ are soluble.
 \item[{\rm (e)}] In the second column of Table \ref{tab:main2} we record the \emph{type} of $H$. If $G_0$ is a classical group with natural module $V$, then this gives an approximate description of the structure of $H \cap {\rm PGL}(V)$ (our usage is consistent with \cite[p.58]{KL}). Note that $P_1$ denotes a parabolic subgroup, which is the stabiliser in $G$ of a $1$-dimensional totally isotropic subspace of $V$. For the cases with $G_0 = {}^2F_4(2)'$, the type of $H$ coincides with the structure of $H \cap G_0$.
\item[{\rm (f)}] Consider the case recorded in the third row of Table \ref{tab:main2}. First note that there are two groups of the form $G_0.f$, namely $G_0.\la \phi \ra$ and $G_0.\la \delta\phi\ra$, where $\phi$ is a field automorphism of order $f$ and $\delta$ is a diagonal automorphism. In addition we require $q=2r^a - 1$, where $r = 2^{m}+1$ is a Fermat prime, $m \geqs 2$ is a $2$-power, $a$ is a positive integer and $f=2^{m-1}$. See Remark \ref{r:nt} for further comments on the number-theoretic conditions arising in the second and third rows of Table \ref{tab:main2}.
   \item[{\rm (g)}] In the final column of Table \ref{tab:main2} we describe the unique conjugacy class of derangements of prime order in $G$. If $G_0 = {\rm U}_{3}(3)$ with $H$ of type ${\rm L}_{2}(7)$, then $G_0$ contains two $G$-classes of elements of order $3$; as indicated in the table, the derangements have Jordan form $[J_2,J_1]$ on $V$, where $J_i$ denotes a standard unipotent Jordan block of size $i$. Similarly, if $G_0 = {}^2F_4(2)'$ and $H$ is of type ${\rm L}_{2}(25)$, then $G_0$ has two $G$-classes of involutions, labelled \texttt{2A} and \texttt{2B} with $|\texttt{2A}| = 1755$ and $|\texttt{2B}| = 11700$; the derangements are in \texttt{2A}. In each of the remaining cases, we use a prime $\ell$ to describe the unique class of derangements of prime order; in every case, this is the unique $G$-class of elements of order $\ell$ in $G_0$. In the first row, for instance, ${\rm PGL}_{2}(q)$ has two classes of involutions, with representatives labelled $t_1$ and $t_1'$ in \cite[Table 4.5.1]{GLS}; since $q \equiv 3 \imod {4}$, the involutions of type $t_1'$ are contained in $G_0$ and they are the only  derangements of prime order in $G$.
\end{itemize}
\end{remk}

The analysis of the primitive almost elusive permutation groups initiated in this paper has recently been extended in \cite{Hall1}, where the almost simple classical groups are handled. A classification of the almost elusive primitive groups will be presented in \cite{Hall2}, which will play a key role in completing the full classification in the more general quasiprimitive setting. 

\vs

\noindent \textbf{Notation.} The notation we adopt in this paper is all fairly standard. Let $A$ and $B$ be groups and let $n$ be a positive integer. We write $C_n$, or just $n$, for a cyclic group of order $n$ and $A^n$ is the direct product of $n$ copies of $A$. An unspecified extension of $A$ by $B$ will be denoted by $A.B$ and we use $A{:}B$ if the extension splits. We adopt the standard notation for simple groups from \cite{KL}. For positive integers $a$ and $b$, we write $(a,b)$ for the greatest common divisor of $a$ and $b$. 

\vs

\noindent \textbf{Acknowledgments.} Both authors thank Tim Dokchitser and Michael Giudici for helpful conversations concerning the content of this paper. EVH also acknowledges the financial support of the Heilbronn Institute for Mathematical Research.

\section{A reduction theorem}\label{s:red}

In this section we prove Theorem \ref{t:main1}. Let $G \leqs {\rm Sym}(\O)$ be a finite quasiprimitive almost elusive group with point stabiliser $H$ and socle $N$. By \cite[Theorem 1]{P93} we have $N = T_1 \times \cdots \times T_k$, where $k \geqs 1$ and each $T_i$ is isomorphic to a fixed simple group $T$. Note that $G = NH$ since $N$ is transitive. Let $\pi_i : N \to T_i$, $i = 1, \ldots, k$, be the natural projection maps.

First assume $N$ is abelian, so $N = (C_p)^k$ for some prime $p$. Here $N$ is regular and \cite[Theorem 1]{P93} implies that $G$ is an affine group. Moreover, each nontrivial element in $N$ is a derangement, so the almost elusivity of $G$ implies that $H$ acts transitively on these elements and thus $G$ is $2$-transitive.

For the remainder, we may assume $N$ is non-abelian. If $k=1$ then $G$ is almost simple, so we may assume $k \geqs 2$. Let $J$ be a minimal normal subgroup of $G$ and note that $N = J \times C_G(J)$ (see the proof of \cite[Theorem 1]{P93}). If $C_G(J) \ne 1$ then both $J$ and $C_G(J)$ are regular on $\O$ and thus every nontrivial element in $J$ is a derangement. But $|T|$ is divisible by at least three distinct primes, which implies that $G$ contains at least three conjugacy classes of derangements of prime order. This is a contradiction. Therefore, $C_G(J) = 1$ and $N=J$ is a minimal normal subgroup. In particular, $H$ acts transitively on the set $\{T_1, \ldots, T_k\}$ and it follows that there exists a subgroup $R \leqs T$ such that $\pi_i(H \cap N) \cong R$ for all $i$. We now consider two separate cases.

First assume $R=T$. Here $H \cap N = D_1 \times \cdots \times D_l \cong T^{l}$, where each 
\[
D_i = \{(x,x^{\varphi_{i,1}}, \ldots, x^{\varphi_{i,m-1}}) \,: \, x \in T \} \cong T
\]
is a full diagonal subgroup of $\prod_{j \in I_i}T_j$ and the $I_i$ partition $\{1, \ldots, k\}$ (here each $\varphi_{i,j}$ is an automorphism of $T$). Note that $k =lm$ and $m \geqs 2$. Clearly, we have $T_1 \cap H = 1$, so each nontrivial element in $T_1$ is a derangement on $\O$ and as above we deduce that $G$ contains at least three conjugacy classes of derangements of prime order. Once again, this is a contradiction.

Finally, let us assume $R<T$. Here we are in Case 2(b) in the proof of \cite[Theorem 1]{P93} and it follows that $G \leqs L \wr S_k$ in its natural product action on $\Delta = \Gamma^k$, where $L \leqs {\rm Sym}(\Gamma)$ is a quasiprimitive almost simple group with socle $T$ and point stabiliser $U$ (note that $T$ acts transitively on $\Gamma$ since $L$ is quasiprimitive). In particular, $G$ is a group of type III(b)(i) in the notation of \cite[Section 2]{P93}, which means that the following hold:
\begin{itemize}\addtolength{\itemsep}{0.2\baselineskip}
\item[{\rm (a)}] $N = T^k$ is the unique minimal normal subgroup of $G$.
\item[{\rm (b)}] $\Delta$ is a $G$-invariant partition of $\O$.
\item[{\rm (c)}] Fix $\gamma \in \Gamma$ and $\delta = (\gamma, \ldots, \gamma) \in \Delta$. If $\a \in \O$ is contained in the part $\delta \in \Delta$, then $N_{\delta} = (T_{\gamma})^k$ and $N_{\a}$ is a subdirect product of $S^k$ for some nontrivial normal subgroup $S$ of $T_{\gamma}$.
\end{itemize}
In particular, there exist $\a \in \O$ and $\gamma \in \Gamma$ such that 
\[
N_{\a} \leqs (T_{\gamma})^k < T^k = N.
\]

If $z \in T$ is a derangement of prime order with respect to the action of $T$ on $\Gamma$, then the elements $(z,1, \ldots, 1)$ and $(z,z,1, \ldots, 1)$ in $N$ are derangements of prime order on $\O$. Moreover, these elements are not $G$-conjugate and thus $G$ is not almost elusive.

Therefore, to complete the proof, we may assume that $T$ is elusive on $\Gamma$. By applying \cite[Theorem 1.4]{G} we see that $L=T = {\rm M}_{11}$ and $U = {\rm L}_{2}(11)$. Since $U$ is simple, property (c) above implies that $N_{\a}$ is a subdirect product of $U^k$. If $N_{\a} = U^k$ then $N$ is elusive and by arguing as in the proof of \cite[Theorem 1.1]{G} we deduce that $G = {\rm M}_{11} \wr A$ for some transitive subgroup $A \leqs S_k$. But then $G$ is elusive and we have reached a contradiction. Finally, suppose $N_{\a}<U^k = U_1 \times \cdots \times U_k$ and write $N_{\a} = F_1 \times \cdots \times F_c$, where each $F_i \cong U$ is a full diagonal subgroup of $\prod_{j \in I_i}U_j$  and the $I_i$ partition $\{1, \ldots, k\}$. Then by arguing as above (the case $R=T$) we deduce that $G$ contains at least three classes of derangements of prime order. This final contradiction completes the proof of Theorem \ref{t:main1}.

\section{Symmetric and alternating groups}\label{s:sym}

In this section, we begin the proof of Theorem \ref{t:main2} by considering the almost simple groups with socle an alternating group. Our main result is the following.

\begin{thm}\label{t:sym}
Let $G \leqs {\rm Sym}(\O)$ be an almost simple primitive permutation group with socle $G_0 = A_n$ and point stabiliser $H$. Then $G$ is almost elusive if and only if $(G,H)$ is one of the cases recorded in Table \ref{tab:main1}. 
\end{thm}

It is straightforward to handle the cases with $n \leqs 10$.

\begin{prop}\label{p:small}
The conclusion to Theorem \ref{t:sym} holds if $n \leqs 10$.
\end{prop}

\begin{proof}
This is an entirely straightforward {\sc Magma} \cite{magma} calculation. In each case, we use the functions \texttt{MaximalSubgroups} and \texttt{CosetAction} to construct $G$ as a permutation group on the set of cosets of $H$. Then by taking a set of conjugacy class representatives in $G$, we can read off the derangements of prime order and verify the result.
\end{proof}

For the remainder, we may assume $G = S_n$ or $A_n$ with $n \geqs 11$. We will divide the rest of the proof into three parts, according to the action of $H$ on $\{1, \ldots, n\}$. We denote the cycle-shape of an element $g \in S_n$ of prime order $r$ by writing $[r^d,1^{n-dr}]$, where $d$ is the number of $r$-cycles in the cycle decomposition of $g$.

\subsection{Intransitive subgroups}\label{ss:intrans}

We start by assuming $H$ acts intransitively on $\{1, \ldots, n\}$. Therefore 
$H = (S_k \times S_{n-k}) \cap G$ and we may identify $\O$ with the set of $k$-element subsets ($k$-sets for short) of $\{1, \ldots, n\}$ for some $k$ in the range $1 \leqs k < n/2$.

We will need some number-theoretic preliminaries on the prime factors of $|\O| = \binom{n}{k}$.

\begin{lem}\label{l:EE_1}
If $|\O|$ is divisible by a prime power $p^a$, then $p^a\leqslant n$.
\end{lem}

\begin{proof}
See \cite[Lemma, p.1084]{EE_72}.
\end{proof}

\begin{lem}\label{l:EEES_1}
Write $|\O|=UV$, where $U=p_1^{a_1}\cdots p_l^{a_l}$, $V=q_1^{b_1}\cdots q_m^{b_m}$ and $p_i, q_j$ are distinct primes such that $p_i<k$ and $q_i\geqslant k$ for all $i$. Then either 
\begin{itemize}\addtolength{\itemsep}{0.2\baselineskip}
    \item[{\rm (i)}] $U \leqs V$; or
    \item[{\rm (ii)}] $(n,k) = (8,3), (9,4), (10,5),(12,5),(21,7),(21,8),(30,7), (33,13),(33,14),(36,13)$, $(36,17)$ or $(56,13)$.
\end{itemize}
\end{lem}

\begin{proof}
This is \cite[Theorem, p.258]{EEES_78}.
\end{proof}

\begin{lem}\label{l:bd}
Suppose $n \geqs 12$ and $k$ is a prime such that $5\leqslant k<\frac{n}{2}$. Then $\binom{n}{k}>n^4$ if $k \geqs 11$, or if $k=7$ and $n \geqs 24$, or $k=5$ and $n \geqs 130$.
\end{lem}

\begin{proof}
This is an easy exercise and we omit the details.
\end{proof}

A classical theorem of Sylvester and Schur (see \cite[p.258]{EEES_78}) states that $|\O|$ is divisible by a prime $r>k$. For $k \geqs 4$, we can now establish the following extension.

\begin{prop}\label{binom_div_prime}
For $k \geqs 4$, either $|\O|$ is divisible by distinct primes $r,s>k$, or $(n,k) = (12,5), (9,4)$.
\end{prop}

\begin{proof}
Write $|\O|=UV$ as in the statement of Lemma \ref{l:EEES_1}. Our aim is to show that $V$ has at least two distinct prime divisors $q_1$ and $q_2$ that are not equal to $k$. This is clear if $m \geqs 3$. Let us also note that the cases arising in part (ii) of the lemma can be checked using {\sc Magma}; the only exceptions are $\binom{12}{5}$ and $\binom{9}{4}$. For the remainder, we may assume $U\leqslant V$ and $m \leqs 2$. 

First assume $m=1$, so $V=q_1^{b_1}$. By Lemma \ref{l:EE_1} we have $V\leqslant n$ and thus $|\O|=UV\leqslant V^2\leqslant n^2$. But this is a contradiction since $|\O|>n^2$ for $n\geqslant 9$. 

Now assume $m=2$, so $V=q_1^{b_1}q_2^{b_2}$. Clearly, if $k$ is composite then $q_1,q_2\neq k$ and the result follows. Similarly, if $k$ is a prime and $k$ does not divide $|\O|$, then $q_1,q_2\neq k$ and we are done. Finally, suppose $k$ is a prime divisor of $|\O|$. Set $q_1 = k$, so $V=k^{b_1}q_2^{b_2}$ and $q_2>k$. By Lemma \ref{l:EE_1} we have $k^{b_1},q_2^{b_2}\leqslant n$ and so $V\leqslant n^2$. Since $U \leqs V$ we have $|\O| \leqslant n^4$ and thus Lemma \ref{l:bd} implies that either $k=7$ and $15 \leqs n \leqs 23$, or $k=5$ and $11 \leqs n \leqs 129$. This finite list of cases can be checked using {\sc Magma} and we conclude that $(n,k) = (12,5)$ is the only exception to the main statement of the proposition.
\end{proof}

We will also need the following number-theoretic result, which is \cite[Lemma 2.6]{BTV2}. This lemma will also be useful in Section \ref{s:lie}.

\begin{lem}\label{l:btv}
Let $r$ and $s$ be primes and let $m$ and $n$ be positive integers. If $r^m + 1 =s^n$ then one of the following holds:
\begin{itemize}\addtolength{\itemsep}{0.2\baselineskip}
    \item[{\rm (i)}] $(r,s,m,n)=(2,3,3,2)$.
    \item[{\rm (ii)}] $(r,n)=(2,1)$ and $s=2^m+1$ is a Fermat prime. 
    \item[{\rm (iii)}] $(s,m)=(2,1)$ and $r=2^n-1$ is a Mersenne prime. 
\end{itemize}
\end{lem}

We are now ready to begin the proof of Theorem \ref{t:sym} in the case where $G = S_n$ or $A_n$ and $\O$ is the set of $k$-element subsets of $\{1, \ldots, n\}$ with $1 \leqs k < n/2$.

\begin{lem}\label{4-sets and above}
If $k \geqs 4$ then $G$ is not almost elusive.
\end{lem}

\begin{proof}
Suppose $k \geqs 4$. The cases $(n,k) = (12,5)$ and $(9,4)$ can be handled directly. For example, if $(n,k) = (9,4)$ then it is easy to see that $G$ contains derangements of order $3$ and $7$. In each of the remaining cases, Proposition  \ref{binom_div_prime} implies that $|\O|$ is divisible by at least two distinct primes $r$ and $s$ with $r,s>k$. 

Since $r>k$, it follows that $r$ divides $n-t$ for some $t \in \{0,1,\ldots, k-1\}$ and we can consider an element $g \in G$ with cycle-shape $[r^{(n-t)/r},1^t]$. Since $t<k$, it follows that $g$ is a derangement. Therefore, in the remaining cases we see that $G$ contains derangements of order $r$ and $s$, whence $G$ is not almost elusive.
\end{proof}

\begin{lem}\label{1-sets}
If $k=1$ then $G$ is almost elusive if and only if one of the following holds: 
    \begin{itemize}\addtolength{\itemsep}{0.2\baselineskip}
            \item[{\rm (i)}] $n = r^a$, $r$ prime, with $a \geqs 2$ if $G = A_n$. 
            \item[{\rm (ii)}] $G=A_n$, $n = 2r^a$, $r \geqs 3$ prime.
        \end{itemize}
\end{lem}

\begin{proof}
If $n$ is divisible by two distinct odd primes, say $r$ and $s$, then $G$ contains derangements with cycle-shape $[r^{n/r}]$ and $[s^{n/s}]$, so $G$ is not almost elusive. Therefore, for the remainder we may assume $n = 2^mr^a$, where $r$ is an odd prime and $m,a \geqs 0$. 

Suppose $m,a>0$. If $G = S_n$, or $G=A_n$ with $m\geqs 2$, then elements of the form $[2^{n/2}]$ and $[r^{n/r}]$ are derangements. However, if $G=A_n$ and $m=1$, then $n \equiv 2 \imod{4}$ and $G$ does not contain elements of the form $[2^{n/2}]$, so in this case $G$ is almost elusive. If $a=0$ then $n = 2^m$ and $G$ is almost elusive since both $S_n$ and $A_n$ have a unique conjugacy class of elements with cycle-shape $[2^{n/2}]$. Finally, if $m=0$ then $n = r^a$ and $G$ is almost elusive unless $G = A_n$ and $a=1$, in which case $G$ has two classes of $r$-cycles.
\end{proof}

\begin{lem}\label{2-sets}
If $k=2$ then $G$ is almost elusive if and only if $n=9$, or $G = S_n$ and either $n$ is a Fermat prime, or $n-1$ is a Mersenne prime. 
\end{lem}

\begin{proof}
Let $g \in G$ be an element of order $r$, with cycle-shape $[r^d,1^{n-dr}]$. Clearly, if $r=2$ or $n-dr \geqs 2$, then $g$ fixes a $2$-set. Now assume $r$ is odd and $n - dr \leqs 1$.

First assume $n=2^ml$ is even, where $m\geqslant 1$ and $l$ is odd. If $r$ is a prime divisor of $n-1$ then every element with cycle-shape $[r^{(n-1)/r},1]$ is a derangement, so we may assume $n-1=r^a$ for some $a\geqslant 1$. Similarly, if $r$ is a prime divisor of $l$, then there exist derangements with cycle-shape $[r^{n/r}]$, so we may also assume $n=2^m$. By Lemma \ref{l:btv} we deduce that $a=1$, so $r = 2^m-1$ is a Mersenne prime and $|\O| = 2^{m-1}r$. In particular, every prime order derangement in $G$ is an $r$-cycle and thus $G$ is almost elusive if $G = S_n$, but not if $G=A_n$ (since there are two $A_n$-classes of $r$-cycles).  

Now assume $n = 2^ml +1$ is odd, where $m \geqs 1$ and $l$ is odd. If $r$ is a prime divisor of $n$, then elements of the form $[r^{n/r}]$ are derangements, so we may assume $n = r^a$ is a prime power. Similarly, if $l$ is divisible by an odd prime $s$, then we get derangements of the form $[s^{(n-1)/s},1]$, so we can assume $l=1$ and thus $r^a = 2^m+1$. By Lemma \ref{l:btv}, it follows that either $n=9$, or $n=r = 2^m+1$ is a Fermat prime.

If $n=9$ then it is easy to see that every derangement of prime order has cycle-shape $[3^3]$, so $G$ is almost elusive. Now assume $n = r = 2^m+1$ is a Fermat prime, so $|\O| = 2^mr$ and the only prime order derangements are $r$-cycles. We conclude that $G = S_n$ is almost elusive, but $G = A_n$ has two conjugacy classes of prime order derangements.
\end{proof}

\begin{prop}\label{p:intrans}
The conclusion to Theorem \ref{t:sym} holds if $H$ is intransitive.
\end{prop}

\begin{proof}
We may assume $k=3$ and our aim is to show that $G$ is almost elusive if and only if $n=9$ or $G = A_{10}$. Let $g \in G$ be an element of prime order $r$ with cycle-shape $[r^d,1^{n-dr}]$. Visibly, $g$ is a derangement if and only if $r=2$ and $n=2d$, or $r \geqs 5$ and $n-dr \leqs 2$. We divide the proof into two parts, according to the parity of $n$. Note that the condition $k<n/2$ implies that $n \geqs 7$.

\vs

\noindent \emph{Case 1. $n$ even}

\vs

First assume $n$ is even, say $n=2^ml$ with $m \geqs 1$ and $l \geqs 1$ odd. For now, let us also assume that $m \geqs 2$ if $G = A_n$. Then $G$ contains derangements of shape $[2^{n/2}]$ and the observation above implies that $G$ is almost elusive only if $n=2^m3^b$ and $n-1=3^c$ with  $b,c \geqs 0$. Therefore $n-1=2^m3^b-1=3^c$, so $b=0$ and $n-1=2^m-1=3^c$. But now Lemma \ref{l:btv} implies that $n=4$, so this situation does not arise and we conclude that $G$ is not almost elusive. 

Next assume $G = A_n$ and $n=2l$, where $l \geqs 5$ is odd. If $l$ is divisible by two distinct primes $r,s \geqs 5$, then $G$ is not almost elusive since there are derangements of shape $[r^{n/r}]$ and $[s^{n/s}]$. So we may assume that $l = 3^ar^b$, where $r \geqs 5$ is a prime and $a,b \geqs 0$.

Suppose $b = 0$, so $l = 3^a$, $a \geqs 2$ and we have
\[
|\O| = \binom{n}{3} = 3^{a-1}(n-1)(n-2).
\]
Note that $n-1$ is odd and indivisible by $3$, so it is divisible by a prime $s \geqs 5$ and thus elements in $G$ of shape $[s^{(n-1)/s},1]$ are derangements. If $n-2 = 2^c$, then $3^a-1 = 2^{c-1}$ and Lemma \ref{l:btv} implies that $a=2$ and $c=4$, so $n = 18$. But here $G$ has two classes of $17$-cycles, so $G$ is not almost elusive. Therefore, we have reduced to the case where $n-2$ is divisible by a prime $t \geqs 5$; since $s$ and $t$ are distinct, we conclude that $G$ is not almost elusive.

Now assume $b \geqs 1$, so $G$ contains derangements of shape $[r^{n/r}]$. If $a \geqs 1$, then $n-1$ is divisible by a prime $s \geqs 5$ with $s \ne r$, which implies that $G$ contains derangements of shape $[s^{(n-1)/s},1]$ and thus $G$ is not almost elusive. Now assume $a=0$. Suppose $G$ is almost elusive. Then neither $n-1$ nor $n-2$ can be divisible by a prime $s \geqs 5$, so we have $n-1 = 3^c$ and $n-2 = 2^d3^e$ for integers $c,d$ and $e$. But $n-1$ and $n-2$ are not both divisible by $3$, so $e=0$ and we have $3^c = 2^d+1$. By Lemma \ref{l:btv} we deduce that $c=2$ and $d=3$ is the only solution, so $G = A_{10}$ and this is an almost elusive group (the only derangements have cycle-shape $[5^2]$).  

\vs

\noindent \emph{Case 2. $n$ odd}

\vs

Now assume $n$ is odd, say $n=2^ml+1$ with $m \geqs 1$ and $l$ odd. First assume $n$ is divisible by $3$ and $G$ is almost elusive. Since $n-2$ is odd and indivisible by $3$, it must be divisible by a prime $r \geqs 5$ and thus $G$ contains derangements of shape $[r^{(n-2)/r},1^2]$. Therefore, we must have $n-2 = r^a$. In addition, if $n$ is divisible by a prime $s \geqs 5$, then $s \ne r$ and $G$ contains derangements of the form $[s^{n/s}]$, whence $n = 3^b$. Similarly, $n-1=2^c$ and thus $3^b = 2^c+1$, which has the unique solution $(b,c) = (2,3)$ by Lemma \ref{l:btv}. Therefore, $n = 9$ and every prime order derangement is a $7$-cycle, so both $S_9$ and $A_9$ are almost elusive.

Next assume $n \equiv 1 \imod{3}$, so both $n$ and $n-2$ are odd and indivisible by $3$. Therefore, there exist distinct primes $r,s \geqs 5$ such that $r$ divides $n$ and $s$ divides $n-2$, whence $G$ contains derangements of the form $[r^{n/r}]$ and $[s^{(n-2)/s},1^2]$. In particular, $G$ is not almost elusive.

Finally, suppose $n \equiv 2 \imod{3}$ and $G$ is almost elusive. Let $r\geqs 5$ be a prime divisor of $n$. Then $G$ contains derangements of shape $[r^{n/r}]$, so $n = r^a$. Similarly, if $n-2$ is divisible by a prime $s \geqs 5$, then $G$ contains derangements of the form $[s^{(n-2)/s},1^2]$, so this forces $n-2=3^b$. Similarly, $n-1=2^c$ for some integer $c$ and thus $2^c = 3^b+1$. By Lemma \ref{l:btv} it follows that $(b,c) = (1,2)$ and thus $n=5$, which is a contradiction since $n \geqs 7$.
\end{proof}

\subsection{Imprimitive subgroups}\label{ss:imprim}

Next we assume $H$ acts transitively and imprimitively on $\{1, \ldots, n\}$, so 
$n=ab$ with $a,b \geqslant 2$ and $H = (S_a \wr S_b) \cap G$. In addition, we may identify $\O$ with the set $\O_a^b$ of partitions of $\{1,\dots, n\}$ into $b$ parts of size $a$. In view of Proposition \ref{p:small}, we will assume $n \geqs 11$. 

\begin{lem}\label{l:der1}
Consider the action of $G = S_n$ on $\O = \O_a^b$, where $n \geqs 5$. If $r>a$ is a prime divisor of $|\O|$, then every $r$-cycle in $G$ is a derangement. 
\end{lem}

\begin{proof}
Let $H = S_a \wr S_b$ be a point stabiliser. If $r>b$ then $r$ does not divide $|H|$ and thus every element in $G$ of order $r$ is a derangement. 

Now assume $r\leqslant b$ and let $x\in G$ be an $r$-cycle. Seeking a contradiction, suppose $x$ 
fixes a partition $\a = \{X_1, \ldots, X_b\}$ in $\O$; let $\pi$ be the permutation of $\{1, \ldots, b\}$ induced from the action of $x$ on the parts in $\a$. Note that $\pi \ne 1$ since $r>a$. In fact, since $x$ has order $r$ it follows that $\pi$ also has order $r$ and thus $|{\rm {supp}}(x)|\geqslant ra$ with respect to the action of $x$ on $\{1, \ldots, n\}$. But this is a contradiction since $x$ is an $r$-cycle and $a \geqs 2$. We conclude that $x$ is a derangement.
\end{proof}

Recall \emph{Bertrand's postulate}: for every integer $n \geqs 4$, there exists a prime number in the interval $(n/2,n)$. We will need the following extension, which is a special case of a result due to Ramanujan \cite{Ram}.

\begin{lem}\label{l:ram}
If $n \geqs 12$, then there are at least two primes in the interval $(n/2,n)$.
\end{lem}

\begin{prop}\label{p:imprim}
The conclusion to Theorem \ref{t:sym} holds if $H$ is imprimitive.
\end{prop}

\begin{proof}
As above, write $n=ab$, where $a,b \geqs 2$, and identify $\O$ with the set of partitions of $\{1, \ldots, n\}$ into $b$ subsets of size $a$. By Proposition \ref{p:small}, we may assume $n \geqs 12$. Applying Lemma \ref{l:ram}, fix primes $r,s$ such that $n/2 < r < s < n$. Then $r$ and $s$ both divide $|\O|$ and both primes are strictly larger than $a$, so Lemma \ref{l:der1} implies that every $r$-cycle and every $s$-cycle in $G$ is a derangement. Therefore, $G$ is not almost elusive.
\end{proof}

\subsection{Primitive subgroups}\label{ss:prim}

To complete the proof of Theorem \ref{t:sym}, it remains to handle the groups where $H$ acts primitively on $\{1,\ldots,n\}$. 

\begin{lem}\label{l:bgw}
Let $G \leqs {\rm Sym}(\O)$ be a primitive permutation group with socle $G_0 = A_n$ and point stabiliser $H$. Assume $n \geqs 7$ and $H$ acts primitively on $\{1, \ldots, n\}$.
\begin{itemize}\addtolength{\itemsep}{0.2\baselineskip}
\item[{\rm (i)}] If $r$ is a prime divisor of $|\O|$, then $G$ contains a derangement of order $r$. 
    \item[{\rm (ii)}] $|\O|$ is divisible by at least two distinct primes.
    \end{itemize}
\end{lem}

\begin{proof}
Part (i) is \cite[Proposition 3.5]{BGW}, which follows by combining classical results of Jordan \cite{Jor} and Manning \cite{Man}. Now consider (ii). Seeking a contradiction, suppose $|\O| = r^a$ for some prime $r$. 

First assume $G = A_n$. By \cite[Theorem 1]{Gur2} we have $n = r^a$ and $H \cong A_{n-1}$, so \cite[Lemma 2.2]{Wil2} implies that $H$ is the stabiliser of a point in the natural action of $\{1, \ldots, n\}$. This is incompatible with the fact that $H$ acts primitively on $\{1, \ldots, n\}$.

Now assume $G = S_n$ and set $L=A_n$. Since $H$ is maximal we have  $H\nleqslant L$ and thus $G=LH$. Therefore, $|L:H\cap L|=r^a$ and so the result for alternating groups implies that $n=r^a$ and $H\cap L = A_{n-1}$ is a point stabiliser with respect to the natural action of $L$ on $\{1, \ldots, n\}$. Write $H \cap L = L_{k} \leqs G_{k}$ for some $k \in \{1, \ldots, n\}$. Since $|H:H\cap L| = 2$ we have  $|H:L_k|=2$ and thus $L_k$ is normal in $H$. In particular, $L_k = L_{k^h}$ for all $h \in H$, so $k = k^h$ for all $h \in H$ and thus $H$ acts intransitively on $\{1, \ldots, n\}$. So once again we have reached a contradiction.
\end{proof}

\begin{prop}\label{p:prim}
The conclusion to Theorem \ref{t:sym} holds if $H$ is primitive.
\end{prop}

\begin{proof}
By Proposition \ref{p:small}, we may assume $n \geqs 11$. Then Lemma \ref{l:bgw} implies that $G$ is not almost elusive.
\end{proof}

\vs

This completes the proof of Theorem \ref{t:main2} for symmetric and alternating groups.

\section{Rank one groups of Lie type}\label{s:lie}

As in Section \ref{s:intro}, let $\mathcal{B}$ be the set of sporadic simple groups, together with the simple groups of Lie type of the form ${\rm L}_{2}(q)$ (with $q \geqs 7$ and $q \ne 9$), ${\rm U}_{3}(q)$ (with $q \geqs 3$), ${}^2G_2(q)$ (with $q \geqs 27$) and ${}^2B_2(q)$ (with $q \geqs 8$); see Remark \ref{r:2}(d) for an explanation of the conditions on $q$ in each case. In this section we will prove the following result, which establishes Theorem \ref{t:main2} in the cases where $G_0 \in \mathcal{B}$ is a group of Lie type (the sporadic groups will be handled in Section \ref{s:spor}).

\begin{thm}\label{t:lie}
Let $G$ be an almost simple primitive permutation group with socle $G_0 \in \mathcal{B}$ and point stabiliser $H$, where $G_0$ is a group of Lie type. Then $G$ is almost elusive if and only if $(G,H)$ is one of the cases recorded in Table \ref{tab:main2}. 
\end{thm}

For the classical groups with socle $G_0 = {\rm L}_{2}(q)$ or ${\rm U}_{3}(q)$, we follow \cite{KL} in referring to the \emph{type} of a maximal subgroup of $G$. Recall that this provides an approximate description of the structure of $H \cap {\rm PGL}(V)$, where $V$ is the natural module for $G_0$. Throughout this section, we set $q=p^f$ with $p$ a prime and we write $H_0 = H \cap G_0$.

Recall that if $n \geqs 2$ is an integer, then a prime divisor $r$ of $q^n-1$ is a \emph{primitive prime divisor} if $q^i-1$ is indivisible by $r$ for all $1 \leqs i < n$. By a well known theorem of Zsigmondy \cite{Zsig}, primitive prime divisors exist unless $(n,q) = (6,2)$, or if $n=2$ and $q$ is a Mersenne prime. Note that if $r$ is such a prime and $m$ is a positive integer, then $r$ divides $q^m-1$ if and only if $n$ divides $m$. Also note that Fermat's Little Theorem implies that $r \equiv 1 \imod{n}$. In addition, it will be useful to observe that every primitive prime divisor of $p^{fn}-1$ is also a primitive prime divisor of $q^n-1$.

Throughout this section, it will be helpful to recall that $G$ is not almost elusive if $|\pi(G) \setminus \pi(H)| \geqs 2$, where $\pi(X)$ denotes the set of prime divisors of $|X|$ (see Remark \ref{r:relusive}). 

\subsection{Two-dimensional linear groups}\label{ss:lin}

In this section we prove Theorem \ref{t:lie} for the groups with socle $G_0 = {\rm L}_{2}(q)$. Write $q=p^f$ with $p$ a prime and set $d = (2,q-1)$. Fix a basis $\{e_1, e_2\}$ for the natural module $V$ and recall that $|G_0| = \frac{1}{d}q(q^2-1)$. As in \cite{KL}, for $g \in {\rm Aut}(G_0)$ we write $\ddot{g}$ for the coset  $G_0g \in {\rm Out}(G_0) = {\rm Aut}(G_0)/G_0$. By \cite[Proposition 2.2.3]{KL} we have
\[
{\rm Out}(G_0) = \left\{ \begin{array}{ll}
\la \ddot{\delta} \ra \times \la \ddot{\phi} \ra = C_2 \times C_f & \mbox{if $p>2$} \\
\la \ddot{\phi} \ra = C_f & \mbox{if $p=2$.}
\end{array}\right.
\]
With respect to the basis $\{e_1,e_2\}$, we may assume $\delta$ is the diagonal automorphism induced by conjugation by  
$\left(\begin{smallmatrix} \mu & 0\\ 0& 1 \end{smallmatrix}\right)$, where $\mathbb{F}_{q}^{\times} = \la \mu \ra$, and $\phi$ is the field automorphism of order $f$ corresponding to the Frobenius map $(a_{ij}) \mapsto (a_{ij}^p)$ on matrices. In particular, we may assume $\phi$ acts on $V$ by sending $ae_1+be_2$ to $a^pe_1 + b^pe_2$. 

The maximal subgroups of $G_0$ were originally determined by Dickson (see Dickson's book \cite{Dickson}, first published in 1901) and the complete list of (core-free) maximal subgroups of $G$ (up to conjugacy) is conveniently reproduced in \cite[Tables 8.1 and 8.2]{BHR}.

\begin{prop}\label{p:max}
Let $G$ be an almost simple group with socle $G_0 = {\rm L}_{2}(q)$ and let $H$ be a core-free maximal subgroup of $G$. Then the type of $H$ is one of the following:
\[
\mbox{$P_1$, ${\rm GL}_{1}(q) \wr S_2$, ${\rm GL}_{1}(q^2)$, ${\rm GL}_{2}(q_0)$ {\rm (}$q=q_0^k$, $k$ prime{\rm )},}
\]
\[
\mbox{$2_{-}^{1+2}.O_2^{-}(2)$ {\rm (}$q=p \geqs 3${\rm )}, $A_5$ {\rm (}$q=p$ or $p^2${\rm )}.}
\]
\end{prop}

\begin{proof}
See Tables 8.1 and 8.2 in \cite{BHR}, which record the precise structure of $H_0$,  together with the exact conditions needed for maximality. For example, we see that
\[
(G,H) = \left\{\begin{array}{ll}
(G_0,S_4) & q = p \equiv \pm 1 \imod{8} \\
(G_0,A_4) & q=p \equiv \pm 3, 5, \pm 13 \imod{40} \\
({\rm PGL}_{2}(q),S_4) & q=p \equiv \pm 11, \pm 19 \imod{40} 
\end{array}\right.
\]
if $H$ is of type $2_{-}^{1+2}.O_2^{-}(2)$. 
\end{proof}

\begin{rem}
Note that if $H$ is of type ${\rm GL}_{1}(q) \wr S_2$ or ${\rm GL}_{1}(q^2)$, then $H_0 = D_{2(q-1)/d}$ or $D_{2(q+1)/d}$, respectively.
\end{rem}

Write ${\rm PGL}_{2}(q) = {\rm GL}_{2}(q)/Z$, where $Z = Z({\rm GL}_{2}(q))$ is the centre of ${\rm GL}_{2}(q)$. We will need to recall some basic properties of certain  conjugacy classes of prime order elements in ${\rm PGL}_{2}(q)$. For a general reference, we refer the reader to \cite[Section 3.2]{BG_book}. Let $x \in {\rm PGL}_{2}(q)$ be an element of prime order $r$ and recall that $x$ is \emph{semisimple} if $r \ne p$ and \emph{unipotent} if $r=p$. Write $x = Z\hat{x}$ with $\hat{x} \in {\rm GL}_{2}(q)$. 

\begin{itemize}\addtolength{\itemsep}{0.2\baselineskip}
    \item[{\rm (a)}] First assume $x$ is semisimple, so $r$ divides $q^2-1$ and $x^{G_0} = x^{{\rm PGL}_{2}(q)}$. Suppose $r$ is odd, which means that we may assume $\hat{x}$ also has order $r$. If $r$ divides $q-1$ then $\hat{x}$ is ${\rm GL}_{2}(q)$-conjugate to a diagonal matrix. On the other hand, if $r$ divides $q+1$ then the eigenvalues of $\hat{x}$  are contained in $\mathbb{F}_{q^2} \setminus \mathbb{F}_q$ and thus $\hat{x}$ acts irreducibly on $V$. In both cases, it will be useful to note that $G$ contains $(r-1)/2$ distinct $G_0$-classes of semisimple elements of order $r$, so there are at least $\lceil (r-1)/2f \rceil$ conjugacy classes in $G$ of such elements.  

 \item[{\rm (b)}] Next assume $x$ is a semisimple involution, so $q$ is odd. Here $x$ is $G_0$-conjugate to either $t_1$ or $t_1'$ in the notation of \cite[Section 3.2]{BG_book}, which is consistent with \cite[Table 4.5.1]{GLS}. These elements are distinguished by the fact that $t_1$ lifts to an involution in ${\rm GL}_{2}(q)$, while $t_1'$ lifts to an irreducible element of order $4$. It is worth noting that $G_0$ has a unique class of semisimple involutions, with $t_1 \in G_0$  if and only if $q \equiv 1 \imod{4}$. In addition, let us record that $t_1$ and $t_1'$ are non-conjugate in ${\rm Aut}(G_0)$. 

 \item[{\rm (c)}] Finally, suppose $x$ is unipotent. Here we may assume $\hat{x}$ has order $p$ and Jordan form $[J_2]$ on $V$. If $p=2$ then $G_0 = {\rm PGL}_{2}(q)$ has a unique conjugacy class of involutions. On the other hand, if $p$ is odd then there are two classes of such elements in $G_0$, which are fused in ${\rm PGL}_{2}(q)$. 
\end{itemize}

We are now ready to begin the proof of Theorem \ref{t:lie} for $G_0 = {\rm L}_{2}(q)$. We start by handling the groups where the underlying field is small.

\begin{prop}\label{p:qsmall}
The conclusion to Theorem \ref{t:lie} holds if $G_0 = {\rm L}_{2}(q)$ and $q \leqs 11$.
\end{prop}

\begin{proof}
This is an entirely straightforward {\sc Magma} \cite{magma} calculation, using the standard commands \texttt{ConjugacyClasses}, \texttt{MaximalSubgroups} and \texttt{CosetAction}. 
\end{proof}

For the remainder, we may assume $q \geqs 13$. The possibilities for the point stabiliser $H$ are recorded in Proposition \ref{p:max} and we consider each one in turn. 

\begin{prop}\label{p:p1}
The conclusion to Theorem \ref{t:lie} holds if $G_0 = {\rm L}_{2}(q)$ and $H$ is of type $P_1$.
\end{prop}

\begin{proof}
Here $H_0 = (C_p)^f{:}C_{(q-1)/d}$ is a Borel subgroup of $G_0$ and we have $|\O| = q+1$. We may identify $\O$ with the set of $1$-dimensional subspaces of the natural module $V$. Notice that if $r$ is an odd prime divisor of $q+1$, then $|H_0|$ is indivisible by $r$ and thus every element in $G_0$ of order $r$ is a derangement. In particular, if $q+1$ is divisible by two distinct odd primes, then $G$ is not almost elusive. So for the remainder, we may assume $q+1 = 2^ar^b$ if $q$ is odd and $q+1 = r^b$ if $q$ is even, where $r$ is an odd prime.

First assume $q$ is even and $q+1 = r^b$. By Lemma \ref{l:btv}, either $q=8$, or $b=1$ and $r$ is a Fermat prime (in which case, $q=2^f$ and $f \geqs 4$ is a $2$-power). The case $q=8$ was handled in Proposition \ref{p:qsmall}, so let us assume $q+1 = r \geqs 17$ is a Fermat prime. As noted above, $G_0$ has $(r-1)/2 = q/2$ distinct conjugacy classes of elements of order $r$ and thus $G$ contains at least $q/2f \geqs 2$ such classes. Since each of these elements is a derangement, we conclude that $G$ is not almost elusive.

Now assume $q$ is odd and $q+1 = 2^ar^b$, where $a \geqs 1$ and $b \geqs 0$. If $b=0$ then Lemma \ref{l:btv} implies that $q = 2^a-1$ is a Mersenne prime, so $|\O| = 2^a$ and $G = G_0$ or ${\rm PGL}_{2}(q)$. In terms of the notation introduced above, each involution in $G_0$ is of type $t_1'$ (since $q \equiv 3 \imod{4}$) and these elements are derangements since they act irreducibly on $V$ (alternatively, note that $|H_0|$ is odd). On the other hand, every $t_1$-type involution in ${\rm PGL}_{2}(q) \setminus G_0$ visibly fixes a $1$-space and we conclude that $G$ is almost elusive.

Finally, let us assume $b \geqs 1$. As noted above, $G$ contains derangements of order $r$. In addition, if $a \geqs 2$ then $q \equiv 3 \imod{4}$ and we note that the involutions in $G_0$ (which are of type $t_1'$) are derangements. Similarly, if  $a=1$ and ${\rm PGL}_{2}(q) \leqs G$ then $G$ contains involutions of type $t_1'$ and these elements are derangements. So to complete the proof, we may assume that $a=1$ and $G \cap {\rm PGL}_{2}(q) = G_0$. Now if $x \in G \setminus G_0$ has prime order, then $x$ is ${\rm PGL}_{2}(q)$-conjugate to a standard field automorphism of the form $\phi^i$ (see \cite[Proposition 4.9.1(d)]{GLS}, for example), where $\phi$ acts on $V$ by sending $ae_1+be_2$ to $a^pe_1+b^pe_2$. In particular, $\phi$ fixes the $1$-space $\la e_1\ra$ and thus  $x$ has fixed points on $\O$. As a consequence, it follows that an element $x \in G$ is a derangement of prime order if and only if $x \in G_0$ has order $r$.

By the theorem of Zsigmondy mentioned at the beginning of Section \ref{s:lie}, there exists a primitive prime divisor $s$ of $p^{2f}-1$. Since we are assuming $r$ is the unique odd prime divisor of $q+1$, it follows that $r=s$ and thus $r \equiv 1 \imod{2f}$, so $r \geqs 2f+1$. If $r> 2f+1$ then $G$ has at least $\lceil (r-1)/2f \rceil \geqs 2$ distinct conjugacy classes of such elements, so $G$ is not almost elusive. Now assume $r = 2f+1$. By arguing as in the proof of \cite[Lemma 4.6]{BTV2} we deduce that $f = 2^m$ is a $2$-power, so $r = 2^{m+1}+1$ is a Fermat prime and thus $m+1 = 2^l$ for some $l \geqs 0$.

If $l=0$, then $f=1$, $r=3$ and $G = {\rm L}_{2}(p)$ is almost elusive since it contains a unique class of elements of order $3$. 

Now assume $l \geqs 1$ and write $G = G_0.J$ with  
\[
J \leqs {\rm Out}(G_0) = \la \ddot{\delta} \ra \times \la \ddot{\phi} \ra = C_2 \times C_f.
\]
Recall that $G_0$ contains $(r-1)/2=f$ distinct conjugacy classes of elements of order $r$. If $J$ does not project onto $\la \ddot{\phi} \ra$, then $G$ has at least two conjugacy classes of elements of order $r$ and thus $G$ is not almost elusive. On the other hand, if this projection is surjective then the condition $G \cap {\rm PGL}_{2}(q) = G_0$ implies that $J = \la \ddot{\phi} \ra$ or $\la \ddot{\delta}\ddot{\phi} \ra$ and we see that $G$ has a unique class of elements of order $r$. We conclude that $G$ is almost elusive if and only if $G = G_0.f$.
\end{proof}

\begin{rem}\label{r:nt}
Consider the case $G_0 = {\rm L}_{2}(q)$ in the proof of Proposition \ref{p:p1}, where $q+1 = 2r^a$, $r = 2^{2^l}+1$ is a Fermat prime and $q=p^f$ with $f = 2^{2^l-1}$. If $l=0$ then $f=1$, $r=3$ and there exist primes $p$ with $p+1 = 2.3^a$ for some $a \geqs 1$. For example, the primes $p<10^6$ of this form are $5$, $17$, $53$, $4373$ and $13121$. For $l=1$ we have $f=2$, $r=5$ and one checks that $3$ and $7$ are the only primes $p<10^6$ with $p^2+1 = 2.5^a$. For $l \geqs 2$, we are not aware of any  solutions to the equation $q+1 = 2r^a$ with $f$ and $r$ as above.
\end{rem}

\begin{prop}\label{p:c2}
The conclusion to Theorem \ref{t:lie} holds if $G_0 = {\rm L}_{2}(q)$ and $H$ is of type ${\rm GL}_{1}(q) \wr S_2$.
\end{prop}

\begin{proof}
Here $H_0 = D_{2(q-1)/d}$ and $|\O| = \frac{1}{2}q(q+1)$. If $r$ is an odd prime divisor of $q+1$ then every element in $G_0$ of order $r$ is a derangement. Therefore, we may assume $q+1 = 2^ar^b$, where $r$ is an odd prime and $a,b \geqs 0$. 

Suppose $q=2^f$ is even, so $2^f+1 = r^b$ and Lemma \ref{l:btv} implies that either $q=8$, or $b=1$, $f$ is a $2$-power and $q+1$ is a Fermat prime. In view of Proposition \ref{p:qsmall}, we can assume we are in the latter situation with $f = 2^m$ and $m \geqs 2$. Here $G$ has at least $q/2f \geqs 2$ conjugacy classes of elements of order $r$, whence $G$ is not almost elusive.

Now assume $q$ is odd and $q+1 = 2^ar^b$, where $a \geqs 1$ and $b \geqs 0$. Here every element in $G_0$ of order $p$ is a derangement, so we may assume $b=0$ and thus $p^f+1=2^a$. By applying Lemma \ref{l:btv}, we deduce that $q = 2^a-1$ is a Mersenne prime and thus $|\O| = 2^{a-1}q$. If $G = G_0$, then $G$ contains two conjugacy classes of elements of order $q$, so $G$ is not almost elusive. On the other hand, if $G = {\rm PGL}_{2}(q)$ then there is a unique class of elements of order $q$ and we observe that both classes of involutions in $G$ have fixed points. Indeed, since $q \equiv 3 \imod{4}$ it follows that the involutions in $H_0 = D_{q-1}$ are of type $t_1'$, while the involution in the centre of $H = D_{2(q-1)}$ is of type $t_1$. It follows that ${\rm PGL}_{2}(q)$ is almost elusive.
\end{proof}

\begin{prop}\label{p:c3}
The conclusion to Theorem \ref{t:lie} holds if $G_0 = {\rm L}_{2}(q)$ and $H$ is of type ${\rm GL}_{1}(q^2)$.
\end{prop}

\begin{proof}
In this case we have $H_0 = D_{2(q+1)/d}$ and $|\O| = \frac{1}{2}q(q-1)$. By arguing as in the proof of the previous proposition, we may assume that $q-1 = 2^ar^b$, where $r$ is an odd prime and $a,b \geqs 0$. 

Suppose $q=2^f$ is even, so $2^f-1=r^b$ and Lemma \ref{l:btv} implies that $b=1$, so $r=2^f-1$ is a Mersenne prime and $f \geqs 5$ is a prime (the case $f=3$ was handled in Proposition \ref{p:qsmall}). Since $G$ contains at least $\lceil (r-1)/2f \rceil \geqs 2$ distinct classes of such elements, we conclude that $G$ is not almost elusive.

Now assume $q$ is odd and $q-1=2^ar^b$ with $a \geqs 1$. Note that every element in $G_0$ of order $p$ is a derangement. If $b \geqs 1$ then $G$ also contains derangements of order $r$, so we may assume $p^f = 2^a+1$. Since the case $q=9$ is excluded (recall that ${\rm L}_{2}(9) \cong A_6$),  Lemma \ref{l:btv} implies that $q=2^a+1 \geqs 17$ is a Fermat prime. If $G = G_0$ then $G$ has two classes of elements of order $q$, so $G$ is not almost elusive. Now assume $G = {\rm PGL}_{2}(q)$ and note that $G$ has a unique class of derangements of order $q$. The involutions in $H_0 = D_{q+1}$ are of type $t_1$ (note that $q \equiv 1 \imod{4}$), while the central involution in $H = D_{2(q+1)}$ is of type $t_1'$. Therefore, every involution in $G$ has fixed points and we conclude that $G$ is almost elusive.
\end{proof}

\begin{prop}\label{p:c5}
The conclusion to Theorem \ref{t:lie} holds if $G_0 = {\rm L}_{2}(q)$ and $H$ is of type ${\rm GL}_{2}(q_0)$, where $q=q_0^k$ with $k$ a prime.
\end{prop}

\begin{proof}
First assume $k$ is odd, so $H_0 = {\rm L}_{2}(q_0)$ and 
\[
|\O| = q_0^{k-1}\left(\frac{q_0^{2k}-1}{q_0^2-1}\right) = q_0^{k-1}\left(\frac{q_0^{k}-1}{q_0-1}\right)\left(\frac{q_0^{k}+1}{q_0+1}\right).
\]
As noted in \cite[Table 8.1]{BHR}, the maximality of $H$ requires $q_0 \ne 2$, so Zsigmondy's theorem \cite{Zsig} implies that there exist primitive prime divisors $r$ and $s$ of $q_0^{2k}-1$ and $q_0^k-1$, respectively. Then $r \ne s$ and both $r$ and $s$ divide $|\O|$, but neither divide $q_0^2-1$. Therefore, every element in $G$ of order $r$ or $s$ is a derangement and we conclude that $G$ is not almost elusive.

Now assume $k=2$, so $H_0 = {\rm PGL}_{2}(q_0)$ and $|\O| = \frac{1}{d}q_0(q+1)$. Suppose $q$ is odd, so $q \equiv 1 \imod{4}$ since $q=q_0^2$. Here $q+1$ is divisible by an odd prime $r$ and we see that every element in $G_0$ of order $r$ is a derangement. Let us also observe that the maximality of $H$ implies that $G \leqs G_0.\la \phi \ra$, where $\phi$ is a field automorphism of order $f$ (see \cite[Table 8.1]{BHR}), so $G$ has two conjugacy classes of unipotent elements of order $p$, whereas $H$ has just one. Therefore, $G$ contains derangements of order $p$ and we deduce that $G$ is not almost elusive.

Finally, let us assume $k=2$ and $q=2^f$ is even. If $r$ is a prime divisor of $q+1$ then every element in $G_0$ of order $r$ is a derangement and so we may assume that $2^f+1 = r^a$. Since $f$ is even, Lemma \ref{l:btv} implies that $r = 2^{f}+1$ is a Fermat prime with $f \geqs 4$ a $2$-power. Finally, since $G$ contains at least $(r-1)/2f \geqs 2$ distinct conjugacy classes of elements of order $r$, we see that $G$ is not almost elusive.
\end{proof}

\begin{prop}\label{p:c6}
The conclusion to Theorem \ref{t:lie} holds if $G_0 = {\rm L}_{2}(q)$ and $H$ is of type $2_{-}^{1+2}.O_2^{-}(2)$.
\end{prop}

\begin{proof}
Here we may assume $q=p \geqs 11$ and by inspecting \cite[Proposition 4.6.7]{KL} we see that $H_0 = S_4$ if $q \equiv \pm 1 \imod{8}$, otherwise $H_0 = A_4$. Every element in $G$ of order $p$ is a derangement and there are two classes of such elements if $G = G_0$, so for the remainder we may assume $G = {\rm PGL}_{2}(q)$ and thus the maximality of $H$ implies that $q \equiv \pm 3 \imod{8}$. In particular, $q$ is neither a Mersenne nor a Fermat prime, whence $q^2-1$ is divisible by a prime $r \geqs 5$ and we deduce that $G$ contains derangements of order $r$. In particular, $G$ is not almost elusive.
\end{proof}

\begin{prop}\label{p:S}
The conclusion to Theorem \ref{t:lie} holds if $G_0 = {\rm L}_{2}(q)$ and $H$ is of type $A_5$.
\end{prop}

\begin{proof}
Here $H_0 = A_5$, $p \geqs 7$ and the maximality of $H$ in $G$ implies that either $G = G_0$, or $q=p^2$ and $G = G_0.\la \phi \ra$, where $\phi$ is an involutory field automorphism (see \cite[Table 8.2]{BHR}). In both cases, $G$ has two conjugacy classes of elements of order $p$ and we deduce that $G$ is not almost elusive.
\end{proof}

\subsection{Three-dimensional unitary groups}\label{ss:uni}

Next we turn to the groups with socle $G_0 = {\rm U}_{3}(q)$, where $q=p^f \geqs 3$. Set $d=(3,q+1)$. 

The cases with $q \leqs 19$ can be handled using {\sc Magma}; as in the proof of Proposition \ref{p:qsmall}, this is a straightforward computation (here it is helpful to recall that $G$ is almost elusive only if $|\pi(G) \setminus \pi(H)| \leqs 1$, where $\pi(X)$ is the set of prime divisors of $|X|$).

\begin{prop}\label{p:qsmall2}
The conclusion to Theorem \ref{t:lie} holds if $G_0 = {\rm U}_{3}(q)$ and $q \leqs 19$.
\end{prop}

In view of the proposition, for the remainder we may assume $q \geqs 23$ and our goal is to prove that $G$ is not almost elusive. The maximal subgroups of $G$ are recorded in \cite[Tables 8.5 and 8.6]{BHR} and by inspection we obtain the following result (see the tables in \cite{BHR} for additional conditions on $G$ and $q$ that are needed for the maximality of $H$).

\begin{prop}\label{p:ma2}
Let $G$ be an almost simple group with socle $G_0 = {\rm U}_{3}(q)$ and let $H$ be a core-free maximal subgroup of $G$. If $q \geqs 23$, then the type of $H$ is one of the following:
\[
\mbox{$P_1$, ${\rm GU}_{2}(q) \times {\rm GU}_{1}(q)$, ${\rm GU}_{1}(q) \wr S_3$, ${\rm GU}_{1}(q^3)$, ${\rm GU}_{3}(q_0)$ {\rm (}$q=q_0^k$, $k \geqs 3$ prime{\rm )},}
\]
\[
\mbox{${\rm SO}_{3}(q)$ {\rm (}$q$ odd{\rm )}, $3^{1+2}.{\rm Sp}_2(3)$ {\rm (}$q=p \equiv 2 \imod{3}${\rm )}, ${\rm L}_{2}(7)$ {\rm (}$q=p${\rm )}, $A_6$ {\rm (}$q=p${\rm ).}}
\]
\end{prop}

The following number-theoretic lemma will be useful.

\begin{lem}\label{l:ppd}
Let $q=p^f$ be a prime power with $q \geqs 3$ and let $\mathcal{P}$ be the set of primitive prime divisors of $q^6-1$. If $\mathcal{P}=\{r\}$ then either $r \geqs 12f+1$, or $q \in \{3,4,5,8,19\}$ and $r=6f+1$.
\end{lem}

\begin{proof}
Suppose $\mathcal{P}=\{r\}$. Since $\mathcal{P}$ contains every primitive prime divisor of $p^{6f}-1$, it follows that $r = 6mf+1$ for some $m \geqs 1$ and so we may assume $r=6f+1$. Note that $r$ divides $q^2-q+1$. If $s \geqs 5$ is a prime divisor of $q^2-q+1$ then it is easy to check that $s$ does not divide $q^2-1$ nor $q^3-1$, so $s$ is a primitive prime divisor of $q^6-1$ and thus $s=r$. Since $q^2-q+1$ is odd and indivisible by $9$, it follows that either $q^2-q+1 = r^e$, or $q \equiv 2 \imod{3}$ and $q^2-q+1 = 3r^e$ for some positive integer $e$. 

Suppose $q \equiv 2 \imod{3}$ and $q^2-q+1 = 3r^e$. If we set $x = -q$ and $y = r$, then we have an integer solution $(x,y)$ to the Diophantine equation $x^2+x+1 = 3y^e$. By a theorem of Nagell \cite{Nagell}, if $e \geqs 3$ then the only integer solutions are $(x,y)=(1,1)$ and $(-2,1)$, neither of which are compatible since $x = -q \leqs -3$. Therefore, $e = 1$ or $2$ and thus 
\[
p^{2f}-p^f+1 = 3(6f+1) \mbox{ or } p^{2f}-p^f+1 = 3(6f+1)^2.
\]
It is straightforward to check that $q=5,8$ are the only possibilities. Note that if $q=5$ then $r=7$ and $q^2-q+1=3r$. Similarly, if $q=8$ then $r=19$ and $q^2-q+1 = 3r$.

Finally, suppose $q \not\equiv 2 \imod{3}$ and $q^2-q+1 = r^e$. Setting $x = -q$ and $y=r$, we get an integer solution to the equation $x^2+x+1 = y^e$. If $e \geqs 2$, then by applying \cite[Proposition 1]{BL} we deduce that $(x,y,e) = (-19,7,3)$ is the only solution. Here $q=19$, $r=7$ and $q^2-q+1 = r^3$. On the other hand, if $e=1$ then $p^{2f}-p^f+1 = 6f+1$ and we find that $q=3$ or $4$. Indeed, if $q=3$ then $r=7$ and $q^2-q+1 = r$. Similarly, if $q=4$ then $r=13$ and $q^2-q+1=r$. The result follows.
\end{proof}

\begin{prop}\label{p:p1_2}
The conclusion to Theorem \ref{t:lie} holds if $G_0 = {\rm U}_{3}(q)$ and $H$ is of type $P_1$, ${\rm GU}_{2}(q) \times {\rm GU}_{1}(q)$, ${\rm GU}_{1}(q) \wr S_3$ or ${\rm SO}_{3}(q)$.
\end{prop}

\begin{proof}
In view of Proposition \ref{p:qsmall2},we may assume $q \geqs 23$. Let $r$ be a primitive prime divisor of $q^6-1$. In each case, we observe that $|H_0|$ is indivisible by $r$ and thus every element in $G_0$ of order $r$ is a derangement. Therefore, we may assume $r$ is the unique primitive prime divisor of $q^6-1$. By applying Lemma \ref{l:ppd} we get $r \geqs 12f+1$, where $q = p^f$ as above, and we note that $G_0$ contains $(r-1)/3 \geqs 4f$ distinct ${\rm PGU}_{3}(q)$-classes of such elements (see \cite[Section 3.3.1]{BG_book}). Since $|{\rm Aut}(G_0): {\rm PGU}_{3}(q)| = 2f$ it follows that there at least $(r-1)/6f \geqs 2$ such classes in $G$ and we conclude that $G$ is not almost elusive.
\end{proof}

\begin{prop}\label{p:c3_2}
The conclusion to Theorem \ref{t:lie} holds if $G_0 = {\rm U}_{3}(q)$ and $H$ is of type ${\rm GU}_{1}(q^3)$. 
\end{prop}

\begin{proof}
Here $H_0 = C_m{:}C_3$ with $m= \frac{1}{d}(q^2-q+1)$, so $|\O| = \frac{1}{3}q^3(q^2-1)(q+1)$. Since $|H_0|$ is odd, it follows that every involution in $G_0$ is a derangement. Similarly, if $p \geqs 5$ then every nontrivial unipotent element in $G_0$ is a derangement. For $p=3$, the unipotent elements with Jordan form $[J_2,J_1]$ are derangements (the elements of order $3$ in $H_0$ have Jordan form $[J_3]$ on the natural module for $G_0$). Finally, suppose $p=2$. In view of Proposition \ref{p:qsmall2} we may assume $q \geqs 32$, which implies that there exists a prime divisor $r$ of $q-1$ with $r \geqs 7$. Since $|H_0|$ is indivisible by $r$, we conclude that every element in $G_0$ of order $r$ is a derangement and the proof of the proposition is complete.
\end{proof}

\begin{prop}\label{p:c5_1}
The conclusion to Theorem \ref{t:lie} holds if $G_0 = {\rm U}_{3}(q)$ and $H$ is of type ${\rm GU}_{3}(q_0)$, where $q=q_0^k$ and $k \geqs 3$ is a prime.
\end{prop}

\begin{proof}
By \cite[Proposition 4.5.3]{KL} we have $H_0 = {\rm U}_{3}(q_0).e$, where $e=3$ if $k=3$ and $q \equiv -1 \imod{9}$, otherwise $e=1$. Let $r$ and $s$ be primitive prime divisors of $q_0^{6k}-1$ and $q_0^k-1$, respectively. Then $|H_0|$ is indivisible by both $r$ and $s$, so every element in $G_0$ of order $r$ or $s$ is a derangement and thus $G$ is not almost elusive.
\end{proof}

\begin{prop}\label{p:c6_2}
The conclusion to Theorem \ref{t:lie} holds if $G_0 = {\rm U}_{3}(q)$ and $H$ is of type $3^{1+2}.{\rm Sp}_2(3)$, ${\rm L}_{2}(7)$ or $A_6$.
\end{prop}

\begin{proof}
In each of these cases we have $q=p$ and so in view of Proposition \ref{p:qsmall2} we may assume that $p \geqs 23$. Then $|H_0|$ is indivisible by $p$ and thus every element in $G_0$ of order $p$ is a derangement. In particular, $G$ is not almost elusive since it contains at least two conjugacy classes of elements of order $p$.
\end{proof}

\subsection{Ree groups}\label{ss:ree}

\begin{prop}\label{p:ree}
If $G_0 = {}^2G_2(q)$ with $q \geqs 27$, then $G$ is not almost elusive. 
\end{prop}

\begin{proof}
Here $q = 3^{2m+1}$ with $m \geqs 1$ and we have $|G_0| = q^3(q^3+1)(q-1)$. The maximal subgroups of $G$ are recorded in \cite[Table 8.43]{BHR}, which is reproduced from \cite{Kleid}. 

First assume $H$ is a Borel subgroup, so $|H_0| = q^3(q-1)$ and $|\O| = q^3+1$. If $r$ is an odd prime divisor of $q^3+1$ then every element in $G_0$ of order $r$ is a derangement (note that $(q^3+1,q-1)=2$). Now $q^3+1$ is divisible by $7$, and it is also divisible by a primitive prime divisor $r$ of $3^{12m+6}-1$. Since $r \geqs 12m+7 \geqs 19$, we deduce that $q^3+1$ is divisible by at least two distinct odd primes and thus $G$ is not almost elusive.

Next suppose $H_0 = 2 \times {\rm L}_{2}(q)$, so $|\O| = q^2(q^2-q+1)$ and $H_0 = C_{G_0}(x)$ for an involution $x \in G_0$. If $r$ is a prime divisor of $q^2-q+1$ then every element in $G_0$ of order $r$ is a derangement. In addition, we observe that there exists an element $y \in G_0$ of order $3$ with $|C_{G_0}(y)| = q^3$ (see \cite[Table 22.2.7]{LS_book}, for example); since $|C_{G_0}(y)|$ is odd, it follows that $y$ is a derangement and we conclude that $G$ is not almost elusive.

Next assume $H_0 = (2^2 \times D_{(q+1)/2}){:}3$. Let $r$ and $s$ be primitive prime divisors of $3^{12m+6}-1$ and $3^{2m+1}-1$, respectively. Then $r,s \geqs 5$ and 
$|H_0|$ is indivisible by $r$ and $s$, whence $G$ is not almost elusive. Similarly, if $H_0 = (q \pm \sqrt{3q}+1){:}6$ and we take $r$ to be any prime divisor $q\mp \sqrt{3q}+1$, then every element in $G_0$ of order $r$ or $s$ is a derangement (with $s$ a primitive prime divisor of $3^{2m+1}-1$ as above).

Finally, let us assume $H_0 = {}^2G_2(q_0)$, where $q=q_0^k$ and $k$ is an odd prime. Let $r$ and $s$ be primitive prime divisors of $q_0^{6k}-1$ and $q_0^k-1$, respectively. Then $|H_0|$ is indivisible by $r$ and $s$, whence $G$ is not almost elusive.
\end{proof}

\subsection{Suzuki groups}\label{ss:suz}

\begin{prop}\label{p:suz}
If $G_0 = {}^2B_2(q)$ then $G$ is not almost elusive. 
\end{prop}

\begin{proof}
This is similar to the proof of the previous proposition. We have $q = 2^{2m+1}$ and $|G_0| = q^2(q^2+1)(q-1)$ with $m \geqs 1$. The maximal subgroups of $G$ are conveniently listed in \cite[Table 8.16]{BHR} (the original reference is \cite{Suz}). It will be useful to observe that 
\[
q^2+1 = (q+\sqrt{2q}+1)(q-\sqrt{2q}+1),
\]
where both factors are odd and coprime. In particular, $q^2+1$ is divisible by at least two distinct odd primes.

First assume $H_0 = q^{1+1}{:}(q-1)$ is a Borel subgroup, so $|\O| = q^2+1$. If $r,s$ are distinct prime divisors of $q^2+1$, then neither prime divides $|H_0|$ and thus $G$ is not almost elusive. The same argument applies if $H_0 = D_{2(q-1)}$. Next assume $H_0 = (q \pm \sqrt{2q}+1){:}4$. Here we take $r$ and $s$ to be prime divisors of $q\mp \sqrt{2q}+1$ and $q-1$, respectively, and we observe that $|H_0|$ is indivisible by both primes. Finally, suppose $H_0 = {}^2B_2(q_0)$, where  $q=q_0^k$, $q_0 \ne 2$ and $k \geqs 3$ is a prime, and let $r$ and $s$ be primitive prime divisors of $q_0^{4k}-1$ and $q_0^k-1$, respectively. Then $r \ne s$ and neither prime divides $|H_0|$, whence all elements in $G_0$ of order $r$ or $s$ are derangements.
\end{proof}

\vs

This completes the proof of Theorem \ref{t:lie}.

\section{Sporadic groups}\label{s:spor}

In this final section we complete the proof of Theorem \ref{t:main2} by handling the almost simple groups with socle a sporadic group. As noted in Section \ref{s:intro}, we also include the almost simple groups with socle ${}^2F_4(2)'$.

\begin{prop}\label{p:2f42}
The conclusion to Theorem \ref{t:main2} holds if $G_0 = {}^2F_4(2)'$.
\end{prop}

\begin{proof}
This is a routine {\sc Magma} computation, using a permutation representation of $G$ of degree $1755$ from the Web-Atlas \cite{WebAt}. 
\end{proof}

\begin{thm}\label{t:spor}
Let $G$ be an almost simple primitive permutation group with socle a sporadic group. Then $G$ is not almost elusive. 
\end{thm}

\begin{proof}
Let $H$ be a point stabiliser and first assume $G \ne \mathbb{B}, \mathbb{M}$, where $\mathbb{B}$ is the Baby Monster and $\mathbb{M}$ is the Monster. In each of these cases we can use the \textsf{GAP} Character Table Library \cite{GAPCTL} to show that $G$ is not almost elusive. Indeed, the character tables of both $G$ and $H$ are available in \cite{GAPCTL} (to access the character table of $H$, we use the \texttt{Maxes} function), together with the fusion map from $H$-classes to $G$-classes. It is now a routine exercise to check that $H$ has at least two conjugacy classes of prime order derangements, with the single exception of the elusive group $G = {\rm M}_{11}$ with $H = {\rm L}_{2}(11)$.

Next assume $G = \mathbb{B}$ and let $\pi(G)$ be the set of prime divisors of $|G|$. Define $\pi(H)$ in the same way. The complete list of maximal subgroups of $G$ (up to conjugacy) is conveniently presented in the Web-Atlas \cite{WebAt} and it is easy to check that $|\pi(G) \setminus \pi(H)| \geqs 2$ in every case. Therefore, we can find distinct primes that divide $|G|$ but not $|H|$, so $G$ contains at least two conjugacy classes of derangements of prime order. 

Finally, let us assume $G = \mathbb{M}$.  There are $44$ known conjugacy classes of maximal subgroups of $G$ and it has been shown that any additional maximal subgroup has to be almost simple, with socle ${\rm L}_{2}(8)$, ${\rm L}_{2}(13)$, ${\rm L}_{2}(16)$ or ${\rm U}_{3}(4)$ (see \cite{Wil}). In every case, including the list of candidate maximal subgroups, one checks that $|\pi(G) \setminus \pi(H)| \geqs 2$ and the result follows as before.
\end{proof}

\end{document}